\documentclass[12pt,abstract]{scrartcl}
\usepackage{graphicx}
\usepackage{amsmath,amsthm}
\makeatletter
\renewcommand{\theenumii}{\@roman\c@enumii}
\makeatother

\usepackage{xcolor}
\usepackage{natbib}
\newtheorem{lemma}{Lemma}
\newtheorem{theorem}{Theorem}

\usepackage{newtxtext,newtxmath}
\begin{document}

\title{Nash equilibrium in asymmetric multi-players zero-sum game with two strategic variables and only one alien}

\author{Atsuhiro Satoh\thanks{atsatoh@hgu.jp}\\[.01cm]
Faculty of Economics, Hokkai-Gakuen University,\\[.02cm]
Toyohira-ku, Sapporo, Hokkaido, 062-8605, Japan,\\[.01cm]
\textrm{and} \\[.1cm]
Yasuhito Tanaka\thanks{yasuhito@mail.doshisha.ac.jp}\\[.01cm]
Faculty of Economics, Doshisha University,\\
Kamigyo-ku, Kyoto, 602-8580, Japan.\\}

\date{}

\maketitle
\thispagestyle{empty}


\begin{abstract}
We consider a partially asymmetric multi-players zero-sum game with two strategic variables. All but one players have the same payoff functions, and one player (Player $n$) does not. Two strategic variables are $t_i$'s and $s_i$'s for each player $i$. Mainly we will show the following results. 1) The equilibrium when all players choose $t_i$'s is equivalent to the equilibrium when all but one players choose $t_i$'s and Player $n$ chooses $s_n$ as their strategic variables. 2) The equilibrium when all players choose $s_i$'s is equivalent to the equilibrium when all but one players choose $s_i$'s and Player $n$ chooses $t_n$ as their strategic variables. The equilibrium when all players choose $t_i$'s and the equilibrium when all players choose $s_i$'s are not equivalent although they are equivalent in a symmetric game in which all players have the same payoff functions.

\end{abstract}

\begin{description}
	\item[Keywords:] partially asymmetric multi-players zero-sum game, Nash equilibrium, two strategic variables
\end{description}

\begin{description}
	\item[JEL Classification:] C72
\end{description}

\section{Introduction}
We consider a multi-players zero-sum game with two strategic variables. Two strategic variables are $t_i$ and $s_i$ for each player $i$. They are related by invertible functions. The game is symmetric for all but one player in the sense that they have the same payoff functions. On the other hand, one player (Player $n$) may have a different payoff function. Thus, the game is partially asymmetric; or there is only one \emph{alien}. In Section 3 we will show the following main results. 
\begin{enumerate}
	\item The equilibrium when all players choose $t_i$'s is equivalent to the equilibrium when all but one players choose $t_i$'s and Player $n$ chooses $s_n$ as their strategic variables.
	\item The equilibrium when all players choose $s_i$'s is equivalent to the equilibrium when all but one players choose $s_i$'s and Player $n$ chooses $t_n$ as their strategic variables.
\end{enumerate}

An example of multi-players zero-sum game with two strategic variables is a relative profit maximization game in an oligopoly with differentiated goods. See Section \ref{ex}. In that section we will show;
\begin{enumerate}
	\item The equilibrium when all players choose $t_i$'s is not equivalent to the equilibrium when all but one players choose $t_i$'s and one player other than Player $n$ chooses $s_i$ as their strategic variables.
	\item The equilibrium when all players choose $t_i$'s is not equivalent to the equilibrium when all but one players choose $s_i$'s and Player $n$ chooses $t_n$ as their strategic variables.
	\item The equilibrium when all players choose $s_i$'s is not equivalent to the equilibrium when all but one players choose $s_i$'s and one player other than Player $n$ chooses $t_i$ as their strategic variables.
	\item The equilibrium when all players choose $s_i$'s is not equivalent to the equilibrium when all but one players choose $t_i$'s and Player $n$ chooses $s_n$ as their strategic variables.
	\item The equilibrium when all players choose $t_i$'s is not equivalent to the equilibrium when all players $s_i$'s.
\end{enumerate}
In these results $t_i$'s are the outputs and $s_i$'s are the prices. In a symmetric game they are all equivalent\footnote{\cite{hst}.}. In Section \ref{ex} we also show that with more than one aliens the equivalence result does not hold.

In the next section we present a model of this paper and prove a preliminary result which is a variation of Sion's minimax theorem. 

\section{The model and the  minimax theorem}

We consider a multi-players zero-sum game with two strategic variables. There are $n$ players, $n\geq 3$. Two strategic variables are $t_i$'s and $s_i$'s, $i\in \{1, \dots, n\}$. We denote $N=\{1, \dots, n\}$. The game is symmetric for Players 1, 2, $\dots$, $n-1$ in the sense that they have the same payoff functions. On the other hand, Player $n$ may have a different payoff function.

$t_i$ is chosen from $T_i$ and $s_i$ is chosen from $S_i$. $T_i$ and $S_i$ are convex and compact sets in linear topological spaces, respectively, for each $i\in \{1, \dots, n\}$. The relations of the strategic variables are represented by
\begin{equation*}
s_i=f_i(t_1, \dots, t_n),\ i\in N,
\end{equation*}
and
\begin{equation*}
t_i=g_i(s_1, \dots, s_n),\ i\in N.
\end{equation*}
$f_i(t_1, \dots, t_n)$ and $g_i(s_1, \dots, s_n)$ are continuous, invertible, one-to-one and onto functions. We assume that all $T_i$'s are identical, and all $S_i$'s are identical. Denote them by $T$ and $S$.

When only Player $n$ chooses $s_n$, then $t_n$ is determined according to
\[t_{n}=g_{n}(f_1(t_1, \dots, t_n),\dots, f_{n-1}(t_1, \dots, t_n), s_n).\]
We denote this $t_n$ by $t_n(t_1, \dots, t_{n-1}, s_n)$.

When all players choose $s_i$'s, $i\in N$, then $t_i$'s for them are determined according to
\[
\left\{
\begin{array}{l}
t_{1}=g_{1}(s_1, \dots, s_n),\\
\dots\\
t_n=g_n(s_1, \dots, s_n).
\end{array}\right.
\]
Denote these $t_i$'s by $t_i(s_1, \dots, s_n)$.

The payoff function of Player $i$ is $u_i,\ i\in N$. It is written as
\[u_i(t_1, \dots, t_n).\]
We assume 
\begin{quote}
$u_i:T_1\times \dots\times T_n\Rightarrow \mathbb{R}$ for each $i\in N$ is continuous on $T_1\times \dots \times T_n$. Thus, it is continuous on $S_1\times \dots \times S_n$ through $f_i,\ i\in N$. It is quasi-concave on $T_i$ and $S_i$ for a strategy of each other player, and quasi-convex on $T_j,\ j\neq i$ and $S_j,\ j\neq i$ for each $t_i$ and $s_i$.
\end{quote}
We do not postulate differentiability of the payoff functions.

Symmetry of the game for Players 1, 2, $\dots$, $n-1$ means that these players are interchangeable in the payoff function of each player. Since the game is a zero-sum game, the sum of the values of the payoff functions of the players is zero.  

Sion's minimax theorem (\cite{sion}, \cite{komiya}, \cite{kind}) for a continuous function is stated as follows.
\begin{lemma}
Let $X$ and $Y$ be non-void convex and compact subsets of two linear topological spaces, and let $f:X\times Y \rightarrow \mathbb{R}$ be a function that is continuous and quasi-concave in the first variable and continuous and quasi-convex in the second variable, then
\[\max_{x\in X}\min_{y\in Y}f(x,y)=\min_{y\in Y}\max_{x\in X}f(x,y).\] \label{l1}
\end{lemma}
We follow the description of Sion's theorem in \cite{kind}.

Applying this lemma to the situation of this paper such that Player $n$ may choose $s_n$ and the other players choose $t_i$'s as their strategic variables, we have the following relations.
\begin{align*}
\max_{t_i\in T}\min_{t_n\in T}u_i(t_i, t_n, \mathbf{t}_k)=\min_{t_n\in T}\max_{t_i\in T}u_i(t_i, t_n, \mathbf{t}_k).
\end{align*}
\begin{align*}
\max_{t_i\in T}\min_{s_n\in S}u_i(t_i, t_n(t_i, s_n,\mathbf{t}_k), \mathbf{t}_k)=\min_{s_n\in S}\max_{t_i\in T}u_i(t_i, t_n(t_i, s_n,\mathbf{t}_k), \mathbf{t}_k),
\end{align*}
where $\mathbf{t}_k$ is a vector of $t_k,\ k\neq i, n$, of the players other than Players $i$ and $n$ who choose $t_k$'s as their strategic variables. $u_i(t_i, t_n, \mathbf{t}_k)$ is the payoff of Player $i$ when Players $i$ and $n$ choose $t_i$ and $t_n$. On the other hand, $u_i(t_i, t_n(t_i, s_n, \mathbf{t}_k), \mathbf{t}_k)$ means the payoff of Player $i$ when he chooses $t_i$ and Player $n$ chooses $s_n$.

We show the following results.
\begin{lemma}
\begin{align*}
&\min_{t_n\in T}\max_{t_i\in T}u_i(t_i, t_n, \mathbf{t}_k)=\min_{s_n\in S}\max_{t_i\in T}u_i(t_i, t_n(t_i, s_n, \mathbf{t}_k), \mathbf{t}_k)\\
&=\max_{t_i\in T}\min_{s_n\in S}u_i(t_i, t_n(t_i, s_n, \mathbf{t}_k), \mathbf{t}_k)=\max_{t_i\in T}\min_{t_n\in T}u_i(t_i, t_n,\mathbf{t}_k),
\end{align*}
\label{l3}
\end{lemma}
\begin{proof}
$\max_{t_i\in T}u_i(t_i,t_n(t_i, s_n, \mathbf{t}_k),\mathbf{t}_k)$ is the maximum of $u_i$ with respect to $t_i$ given $s_n$. Let $\bar{t}_i(s_n)=\arg\max_{t_i\in T}u_i(t_i,t_n(t_i, s_n, \mathbf{t}_k),\mathbf{t}_k)$, and fix the value of $t_n$ at
\begin{equation}
t_n^0=g_n(f_i(\bar{t}_i(s_n), t_n^0, \mathbf{t}_k),\mathbf{f}_k,s_n),\label{tc1}
\end{equation}
where $\mathbf{f}_k$ denotes a vector of the values of $s_k$'s of players who choose $t_k$'s. We have
\begin{align*}
\max_{t_i\in T}u_i(t_i, t_n^0, \mathbf{t}_k)\geq u_i(\bar{t}_i(s_n),t_n^0,\mathbf{t}_k)=\max_{t_i\in T}u_i(t_i,t_n(t_i, s_n, \mathbf{t}_k),\mathbf{t}_k),
\end{align*}
where $\max_{t_i\in T}u_i(t_i, t_n^0,\mathbf{t}_k)$ is the maximum of $u_i$ with respect to $t_i$ given the value of $t_n$ at $t_n^0$. We assume that $\bar{t}_i(s_n)=\arg\max_{t_i\in T}u_i(t_i,t_n(t_i,s_n,\mathbf{t}_k),\mathbf{t}_k)$ is single-valued. By the maximum theorem and continuity of $u_i$, $\bar{t}_i(s_n)$ is continuous, then any value of $t_n^0$ can be realized by appropriately choosing $s_n$ according to (\ref{tc1}). Therefore,
\begin{equation}
\min_{t_n\in T}\max_{t_i\in T}u_i(t_i, t_n,\mathbf{t}_k)\geq \min_{s_n\in S}\max_{t_i\in T}u_i(t_i, t_n(t_i, s_n, \mathbf{t}_k),\mathbf{t}_k).\label{4-11}
\end{equation}

On the other hand, $\max_{t_i\in T}u_i(t_i, t_n,\mathbf{t}_k)$ is the maximum of $u_i$ with respect to $t_i$ given $t_n$. Let $\bar{t}_i(t_n)=\arg\max_{t_i\in T}u_i(t_i, t_n,\mathbf{t}_k)$, and fix the value of $s_n$ at
\begin{equation}
s_n^0=f_n(\bar{t}_i(t_n), t_n,\mathbf{t}_k).\label{sc1}
\end{equation}
We have
\begin{align*}
\max_{t_i\in T}u_i(t_i,t_n(t_i,s_n^0,\mathbf{t}_k),\mathbf{t}_k)\geq u_i(\bar{t}_i(s_n),t_n(t_i,s_n^0,\mathbf{t}_k),\mathbf{t}_k)=\max_{t_i\in T}u_i(t_i, t_n,\mathbf{t}_k),
\end{align*}
where $\max_{t_i\in T}u_i(t_i,t_n(t_i,s_n^0,\mathbf{t}_k),\mathbf{t}_k)$ is the maximum of $u_i$ with respect to $t_i$ given the value of $s_n$ at $s_n^0$. We assume that $\bar{t}_i(t_n)=\arg\max_{t_i\in T}u_i(t_i, t_n,\mathbf{t}_k)$ is single-valued. By the maximum theorem and continuity of $u_i$, $\bar{t}_i(t_n)$ is continuous, then any value of $s_n^0$ can be realized by appropriately choosing $t_n$ according to (\ref{sc1}). Therefore,
\begin{equation}
\min_{s_n\in S}\max_{t_i\in T}u_i(t_i, t_n(t_i,s_n,\mathbf{t}_k),\mathbf{t}_k)\geq \min_{t_n\in T}\max_{t_i\in T}u_i(t_i, t_n,\mathbf{t}_k).\label{4-21}
\end{equation}
Combining (\ref{4-11}) and (\ref{4-21}), we get
\[\min_{s_n\in S}\max_{t_i\in T}u_i(t_i, t_n(t_i,s_n,\mathbf{t}_k),\mathbf{t}_k)=\min_{t_n\in T}\max_{t_i\in T}u_i(t_i, t_n,\mathbf{t}_k).\]
Since any value of $s_n$ can be realized by appropriately choosing $t_n$, we have
\begin{equation*}
\min_{s_n\in S}u_i(t_i,t_n(t_i,s_n,\mathbf{t}_k),\mathbf{t}_k)=\min_{t_n\in T}u_i(t_i, t_n,\mathbf{t}_k).
\end{equation*}
Thus,
\[\max_{t_i\in T}\min_{s_n\in S}u_i(t_i,t_n(t_i,s_n,\mathbf{t}_k),\mathbf{t}_k)=\max_{t_i\in T}\min_{t_n\in T}u_i(t_i, t_n,\mathbf{t}_k).\]
Therefore,
\begin{align*}
&\min_{t_n\in T}\max_{t_i\in T}u_i(t_i, t_n,\mathbf{t}_k)=\min_{s_n\in S}\max_{t_i\in T}u_i(t_i, t_n(t_i, s_n, \mathbf{t}_k),\mathbf{t}_k),\\
=&\max_{t_i\in T}\min_{s_n\in S}u_i(t_i, t_n(t_i,s_n,\mathbf{t}_k),\mathbf{t}_k)=\max_{t_i\in T}\min_{t_n\in T}u_i(t_i, t_n,\mathbf{t}_k).
\end{align*}
\end{proof}

\section{The main results}

In this section we present the main result of this paper.
\begin{theorem}
The equilibrium where all players choose $t_i$'s is equivalent to the equilibrium when one player (Player $n$) chooses $s_n$ and all other players choose $t_i$'s as their strategic variables.\label{t1}
\end{theorem}
\begin{proof}
\begin{enumerate}
	\item Consider a situation $(t_1,\dots,t_{n-1},t_n)=(t,\dots,t, t_n)$. By symmetry for Players 1, 2, $\dots$, $n-1$
\[\max_{t_i\in T}u_i(t,\dots,t_i,\dots,t, t_n)=\max_{t_j\in T}u_j(t,\dots,t_j,\dots,t, t_n),\ \mathrm{for\ any\ }i,j\neq n,\]
and
\[\arg\max_{t_i\in T}u_i(t,\dots,t_i,\dots,t, t_n)=\arg\max_{t_j\in T}u_j(t,\dots,t_j,\dots,t, t_n)\in T,\ \mathrm{for\ any\ }i,j\neq n,\]
given $t_n$.

Let
\[t_n(t)=\arg\max_{t_n\in T}u_n(t,\dots,t, t_n).\]
We assume that it is a single-valued continuous function. 

Consider the following function.
\[t\rightarrow \arg\max_{t_i\in T}u_i(t,\dots,t_i,\dots,t, t_n),\ \mathrm{for\ any\ }i\neq n,\ \mathrm{given}\ t_n.\]
This function is continuous and $T$ is compact. Thus, there exists a fixed point given $t_n$. Denote it by $t^*(t_n)$, then
\[t^*(t_n)=\arg\max_{t_i\in T}u_i(t^*(t_n),\dots,t_i,\dots,t^*(t_n), t_n),\ \mathrm{for\ any\ }i\neq n,\ \mathrm{given}\ t_n.\]
Now we consider the following function.
\[t\rightarrow t^*(t_n(t)).\]
This also has a fixed point. Denote it by $t^*$ and $t_n(t^*)$ by $t_n^*$, then we have
\[t^*=\arg\max_{t_i\in T}u_i(t^*,\dots,t_i,\dots,t^*,t_n^*),\ \mathrm{for\ any\ }i\neq n,\]
\[t_n^*=\arg\max_{t_n\in T}u_n(t^*,\dots,t^*,t_n),\]
\[\max_{t_i\in T}u_i(t^*,\dots,t_i,\dots,t^*,t_n^*)=u_i(t^*, \dots, t^*, t_n^*),\ \mathrm{for\ any\ }i\neq n,\]
and
\[\max_{t_n\in T}u_n(t^*,\dots,t^*,t_n)=u_n(t^*,\dots,t^*,t_n^*).\]
$(t_1, \dots, t_{n-1}, t_n)=(t^*,\dots,t^*,t_n^*)$ is a Nash equilibrium when all players choose $t_i$'s as their strategic variables.

\item Because the game is zero-sum,
\[u_1(t^*, \dots, t^*, t_n)+\dots +u_{n-1}(t^*, \dots, t^*, t_n)+u_n(t^*,\dots, t^*,t_n)=0.\]
By symmetry for Players 1, 2, $\dots$, $n-1$,
\[(n-1)u_i(t^*, \dots, t^*, t_n)+u_n(t^*, \dots, t^*,t_n)=0.\]
This means
\[(n-1)u_i(t^*, \dots, t^*, t_n)=-u_n(t^*, \dots, t^*,t_n),\]
and
\[(n-1)\min_{t_n\in T}u_i(t^*, \dots, t^*, t_n)=-\max_{t_n\in T}u_n(t^*, \dots, t^*,t_n).\]
From this we get
\[\arg\min_{t_n\in T}u_i(t^*, \dots, t^*, t_n)=\arg\max_{t_n\in T}u_n(t^*, \dots, t^*,t_n)=t_n^*,\ \mathrm{for\ any\ }i\neq n.\]
We have
\[\min_{t_n\in T}u_i(t^*, \dots, t^*, t_n)=u_i(t^*, \dots, t^*, t_n^*)=\max_{t_i\in T}u_i(t^*, \dots, t_i, \dots, t^*, t_n^*),\ \mathrm{for\ any\ }i\neq n.\]
Thus,
\begin{align*}
&\min_{t_n\in T}\max_{t_i\in T}u_i(t^*, \dots, t_i, \dots, t^*, t_n)\leq \max_{t_i\in T}u_i(t^*, \dots, t_i, \dots, t^*, t^*_n)=\min_{t_n\in T}u_i(t^*, \dots, t^*, t_n)\\
&\leq \max_{t_i\in T}\min_{t_n\in T}u_i(t^*, \dots, t_i, \dots, t^*, t_n),\ \mathrm{for\ any\ }i\neq n.
\end{align*}
From Lemma \ref{l3} we obtain
\begin{align}
&\min_{t_n\in T}\max_{t_i\in T}u_i(t^*, \dots, t_i, \dots, t^*, t_n)=\max_{t_i\in T}u_i(t^*, \dots, t_i, \dots, t^*, t_n^*)\label{l3-1}\\
&=\min_{t_n\in T}u_i(t^*, \dots, t^*, t_n)=\max_{t_i\in T}\min_{t_n\in T}u_i(t^*, \dots, t_i, \dots, t^*, t_n)\notag\\
&=\min_{s_n\in S}\max_{t_i\in T}u_i(t^*, \dots, t_i, \dots, t^*, t_n(t^*, \dots, t_i, \dots, t^*,s_n))\notag\\
&=\max_{t_i\in T}\min_{s_n\in S}u_i(t^*, \dots, t_i, \dots, t^*, t_n(t^*, \dots, t_i, \dots, t^*,s_n)),\ \mathrm{for\ any\ }i\neq n.\notag
\end{align}

\item Let
\[s^0_n(t^*)=f_n(t^*,\dots, t^*,t_n^*).\]
Since any value of $s_n$ can be realized by appropriately choosing $t_n$,
\begin{equation}
\min_{s_n\in S}u_i(t^*,\dots, t^*,t_n(t^*,\dots,t^*,s_n))=\min_{t_n\in T}u_i(t^*,\dots,t^*,t_n)=u_i(t^*,\dots,t^*,t_n^*).\label{z1}
\end{equation}
Thus,
\[\arg\min_{s_n\in S}u_i(t^*,\dots,t^*,t_n(t^*,\dots,t^*,s_n))=s^0_n(t^*).\]
(\ref{l3-1}) and (\ref{z1}) mean
\begin{align}
&\min_{s_n\in S}\max_{t_i\in T}u_i(t^*,\dots,t_i,\dots,t^*,t_n(t^*,\dots,t_i,\dots,t^*,s_n)) \label{z2} \\
&=\min_{s_n\in S}u_i(t^*,\dots,t^*,t_n(t^*,\dots,t^*,s_n)).\notag
\end{align}
We have
\[\max_{t_i\in T}u_i(t^*,\dots,t_i,\dots,t^*,t_n(t^*,\dots,t_i,\dots,t^*,s_n))\geq u_i(t^*,\dots,t^*,t_n(t^*,\dots,t^*,s_n)).\]
Thus,
\begin{align*}
&\arg\min_{s_n\in S}\max_{t_i\in T}u_i(t^*,\dots,t_i,\dots,t^*,t_n(t^*,\dots,t_i,\dots,t^*,s_n))\\
&=\arg\min_{s_n\in S}u_i(t^*,\dots,t^*,t_n(t^*,\dots,t^*,s_n))=s^0_n(t^*).
\end{align*}
By (\ref{z2})
\begin{align*}
&\min_{s_n\in S}\max_{t_i\in T}u_i(t^*,\dots,t_i,\dots,t^*,t_n(t^*,\dots,t_i,\dots,t^*,s_n))\\
&=\max_{t_i\in T}u_i(t^*,\dots,t_i,\dots,t^*,t_n(t^*,\dots,t_i,\dots,t^*,s_n^0(t^*)))\\
=&\min_{s_n\in S}u_i(t^*,\dots,t^*,t_n(t^*,\dots,t^*,s_n))=u_i(t^*,\dots,t^*,t_n(t^*,\dots,t^*,s^0_n(t^*))).
\end{align*}
Therefore,
\begin{equation}
\arg\max_{t_i\in T}u_i(t^*,\dots,t_i,\dots,t^*,t_n(t^*,\dots,t_i,\dots,t^*,s^0_n(t^*))=t^*.\label{t1-2}
\end{equation}
This holds for any player $i\neq n$.

On the other hand, because any value of $s_n$ is realized by appropriately choosing $t_n$,
\[\max_{s_n\in S}u_n(t^*,\dots,t^*,t_n(t^*,\dots,t^*,s_n))=\max_{t_n\in T}u_n(t^*,\dots,t^*,t_n)=u_n(t^*,\dots,t^*,t_n^*).\]
Therefore,
\begin{equation}
\arg\max_{s_n\in S}u_n(t^*,\dots,t^*,t_n(t^*,\dots,t^*,s_n))=s^0_n(t^*)=f_C(t^*,\dots,t^*,t_n^*).\label{t1-1}
\end{equation}

From (\ref{t1-2}) and (\ref{t1-1}), $(t^*,\dots,t^*,t_n(t^*,\dots,t^*,s^0_n(t^*)))$ is a Nash equilibrium which is equivalent to $(t^*,\dots,t^*,t_n^*)$. 
\end{enumerate}

\end{proof}

Interchanging $t_i$ and $s_i$ for each player, we can show 
\begin{theorem}
The equilibrium where all players choose $s_i$'s is equivalent to the equilibrium when one player (Player $n$) chooses $t_n$ and all other players choose $s_i$'s as their strategic variables.\label{t2}
\end{theorem}

\section{Various examples}\label{ex}

Consider a game of relative profit maximization under oligopoly including four firms with differentiated goods\footnote{%
About relative profit maximization in an oligopoly see \cite{mm}, \cite{ebl2}, \cite{eb2}, \cite{st}, \cite{eb1}, \cite{ebl1} and  \cite{redondo}}. It is a four-players zero-sum game with two strategic variables. The firms are A, B, C and D. The strategic variables are the outputs and the prices of their goods. We consider the following four patterns of competition. 

\begin{enumerate}
\item Pattern 1: All firms determine their outputs. It is a Cournot case.

The inverse demand functions are 
\begin{equation*}
p_A=a-x_A-bx_B-bx_C-bx_D,
\end{equation*}
\begin{equation*}
p_B=a-x_B-bx_A-bx_C-bx_D,
\end{equation*}
\begin{equation*}
p_C=a-x_C-bx_A-bx_B-bx_D,
\end{equation*}
and
\begin{equation*}
p_D=a-x_D-bx_A-bx_B-bx_C,
\end{equation*}
where $0<b<1$. $p_A$, $p_B$, $p_C$ and $p_D$ are the prices of the goods of Firms A, B, C and D, and $x_A$, $x_B$, $x_C$ and $x_D$ are their outputs.

\item Pattern 2: Firms A, B and C determine their outputs, and Firm D determines the price of its good.

From the inverse demand functions, 
\[p_A=(1-b)a+b^2x_C-bx_C+b^2x_B-bx_B+b^2x_A-x_A+bp_D,\]
\[p_B=(1-b)a+b^2x_C-bx_C+b^2x_B-x_B+b^2x_A-bx_A+bp_D,\]
\[p_C=(1-b)a+b^2x_C-x_C+b^2x_B-bx_B+b^2x_A-bx_A+bp_D,\]
\[x_D=a-bx_C-bx_B-bx_A-p_D\]
are derived.

\item Pattern 3: Firms A, B and C determine the prices, and Firm D determines the output.

From the inverse demand functions, 
\[x_A=\frac{(1-b)a+b^2x_D-bx_D+bp_C+bp_B-bp_A-p_A}{(1-b)(2b+1)},\]
\[x_B=\frac{(1-b)a+b^2x_D-bx_D+bp_C-bp_B-p_B+bp_A}{(1-b)(2b+1)},\]
\[x_C=\frac{(1-b)a+b^2x_D-bx_D-bp_C-p_C+bp_B+bp_A}{(1-b)(2b+1)},\]
\[p_D=\frac{(1-b)a+3b^2x_D-2bx_D-x_D+bp_C+bp_B+bp_A}{2b+1}.\]

\item Pattern 4: All firms determine the prices. It is a Bertrand case.

From the inverse demand functions, the direct demand functions are derived as
follows; 
\[x_A=\frac{(1-b)a+bp_D+bp_C+bp_B-2bp_A-p_A}{(1-b)(3b+1)},\]
\[x_B=\frac{(1-b)a+bp_D+bp_C-2bp_B-p_B+bp_A}{(1-b)(3b+1)},\]
\[x_C=\frac{(1-b)a+bp_D-2bp_C-p_C+bp_B+bp_A}{(1-b)(3b+1)},\]
\[x_D=\frac{(1-b)a-2bp_D-p_D+bp_C+bp_B+bp_A}{(1-b)(3b+1)}.\]

\end{enumerate}

The absolute profits of the firms are 
\begin{equation*}
\pi_A=p_Ax_A-c_Ax_A,
\end{equation*}
\begin{equation*}
\pi_B=p_Bx_B-c_Bx_B,
\end{equation*}
\begin{equation*}
\pi_C=p_Cx_C-c_Cx_C,
\end{equation*}
and
\begin{equation*}
\pi_D=p_Dx_D-c_Dx_D.
\end{equation*}
$c_A$, $c_B$, $c_C$ and $c_D$ are the constant marginal costs of Firms A, B, C and D. The relative profits of the firms are 
\begin{equation*}
\varphi_A=\pi_A-\frac{\pi_B+\pi_C+\pi_D}{3},
\end{equation*}
\begin{equation*}
\varphi_B=\pi_B-\frac{\pi_A+\pi_C+\pi_D}{3},
\end{equation*}
\begin{equation*}
\varphi_C=\pi_C-\frac{\pi_A+\pi_B+\pi_D}{3},
\end{equation*}
and
\begin{equation*}
\varphi_D=\pi_D-\frac{\pi_A+\pi_B+\pi_C}{3}.
\end{equation*}

The firms determine the values of their strategic variables to maximize the
relative profits. We see 
\begin{equation*}
\varphi_A+\varphi_B+\varphi_C+\varphi_D=0,
\end{equation*}
so the game is zero-sum. We assume $c_A=c_B=c_C$, that is, the game is symmetric for Firms A, B and C. However, $c_D$ is not equal to $c_A$. Thus, the game is partially asymmetric. Firm D is an alien.

We calculate the equilibrium outputs of the firms in the above four patterns.
\begin{enumerate}
	\item Pattern 1
\[x_A=\frac{bc_D-3c_A-ab+3a}{2(3-b)(b+1)},\]
\[x_B=\frac{bc_D-3c_A-ab+3a}{2(3-b)(b+1)},\]
\[x_C=\frac{bc_D-3c_A-ab+3a}{2(3-b)(b+1)},\]
and
\[x_D=\frac{bc_D-3c_A-ab+3a}{2(3-b)(b+1)}.\]
	\item Pattern 2
\[x_A=\frac{bc_D-3c_A-ab+3a}{2(3-b)(b+1)},\]
\[x_B=\frac{bc_D-3c_A-ab+3a}{2(3-b)(b+1)},\]
\[x_C=\frac{bc_D-3c_A-ab+3a}{2(3-b)(b+1)},\]
and
\[x_D=\frac{bc_D-3c_A-ab+3a}{2(3-b)(b+1)}.\]
	\item Pattern 3
\[x_A=\frac{3b^2c_D+bc_D+4b^2c_A-5bc_A-3c_A-7ab^2+4ab+3a}{2(1-b)(b+1)(7b+3)},\]
\[x_B=\frac{3b^2c_D+bc_D+4b^2c_A-5bc_A-3c_A-7ab^2+4ab+3a}{2(1-b)(b+1)(7b+3)},\]
\[x_C=\frac{3b^2c_D+bc_D+4b^2c_A-5bc_A-3c_A-7ab^2+4ab+3a}{2(1-b)(b+1)(7b+3)},\]
\[x_D=\frac{3a-2b^2c_D-7bc_D-3c_D+9b^2c_A+3bc_A-7ab^2+4ab}{2(1-b)(b+1)(7b+3)}.\]

	\item Pattern 4

\[x_A=\frac{3b^2c_D+bc_D+4b^2c_A-5bc_A-3c_A-7ab^2+4ab+3a}{2(1-b)(b+1)(7b+3)},\]
\[x_B=\frac{3b^2c_D+bc_D+4b^2c_A-5bc_A-3c_A-7ab^2+4ab+3a}{2(1-b)(b+1)(7b+3)},\]
\[x_C=\frac{3b^2c_D+bc_D+4b^2c_A-5bc_A-3c_A-7ab^2+4ab+3a}{2(1-b)(b+1)(7b+3)},\]
\[x_D=\frac{3a-2b^2c_D-7bc_D-3c_D+9b^2c_A+3bc_A-7ab^2+4ab}{2(1-b)(b+1)(7b+3)}.\]
\end{enumerate}
We find that Pattern 1 is equivalent to Pattern 2 (an example of Theorem \ref{t1}), but it is not equivalent to Patterns 3 and 4, and that Pattern 4 is equivalent to Pattern 3 (an example of Theorem \ref{t2}), but it is not equivalent to Patterns 1 and 2.

Next let us examine a case where $c_B=c_A$ and $c_C=c_D$ but $c_A\neq c_D$. Consider the following two patterns of competition.
\begin{enumerate}
	\item Pattern 1: All firms determine their outputs. It is a Cournot case.
	\item Pattern 2: Firms A and B determine their outputs, and Firms C and D determine the prices of their goods, then

\[p_A=\frac{b^2x_B-bx_B+2b^2x_A-bx_A-x_A+bp_D+bp_C-ab+a}{b+1},\]
\[p_B=\frac{2b^2x_B-bx_B-x_B+b^2x_A-bx_A+bp_D+bp_C-ab+a}{b+1},\]
\[x_C=\frac{b^2x_B-bx_B+b^2x_A-bx_A+bp_D-p_C-ab+a}{(1-b)(b+1)},\]
\[x_D=\frac{b^2x_B-bx_B+b^2x_A-bx_A-p_D+bp_C-ab+a}{(1-b)(b+1)}.\]
\end{enumerate}

We calculate the equilibrium outputs of the firms in these two patterns.

\begin{enumerate}
	\item Pattern 1
\[x_A=\frac{2bc_D-bc_A-3c_A-ab+3a}{2(3-b)(b+1)},\]
\[x_B=\frac{2bc_D-bc_A-3c_A-ab+3a}{2(3-b)(b+1)},\]
\[x_C=\frac{3a-bc_D-3c_D+2bc_A-ab}{2(3-b)(b+1)},\]
\[x_D=\frac{3a-bc_D-3c_D+2bc_A-ab}{2(3-b)(b+1)}.\]

	\item Pattern 2

\[x_A=\frac{2bc_D+bc_A-3c_A-3ab+3a}{6(1-b)(b+1)},\]
\[x_B=\frac{2bc_D+bc_A-3c_A-3ab+3a}{6(1-b)(b+1)},\]
\[x_C=\frac{bc_D-3c_D+2bc_A-3ab+3a}{6(1-b)(b+1)},\]
\[x_D=\frac{bc_D-3c_D+2bc_A-3ab+3a}{6(1-b)(b+1)}.\]
\end{enumerate}

Patterns 1 and 2 are not equivalent. Therefore, with more than one aliens the equivalence result does not hold.

\section{Concluding Remarks}

In this paper we have examined equilibria in a partially asymmetric multi-players zero-sum game. We have shown that in an asymmetric zero-sum game with only one alien (a player who has a different payoff function) we get the equivalence result about the choice of strategic variables.

\section*{Acknowledgment}

This work was supported by Japan Society for the Promotion of Science KAKENHI Grant Number 15K03481  and 18K01594.


\begin{thebibliography}{aa}


\bibitem[Hattori, Satoh and Tanaka (2018)]{hst} Hattori, M., Satoh, A., Tanaka, Y., (2018), ``Minimax theorem and Nash equilibrium of symmetric multi-players zero-sum game with two strategic variables,'' Papers 1806.07203, arXiv.org. 

\bibitem[Kindler (2005)]{kind} Kindler, J. (2005), ``A simple proof of Sion's minimax theorem,'' \textit{American Mathematical Monthly}, \textbf{112}, pp. 356-358.

\bibitem[Komiya (1988)]{komiya} Komiya, H. (1988), ``Elementary proof for Sion's minimax theorem,'' \textit{Kodai Mathematical Journal}, \textbf{11},
pp. 5-7.

\bibitem[Matsumura, Matsushima and Cato(2013)]{mm} Matsumura, T., N. Matsushima and S. Cato (2013) ``Competitiveness and R\&D competition revisited,'' \textit{Economic Modelling}, \textbf{31}, pp. 541-547.

\bibitem[Satoh and Tanaka (2013)]{ebl2} Satoh, A. and Y. Tanaka (2013) ``Relative profit maximization and Bertrand equilibrium with quadratic cost functions,'' \textit{Economics and Business Letters},  \textbf{2}, pp. 134-139.

\bibitem[Satoh and Tanaka (2014a)]{eb2} Satoh, A. and Y. Tanaka (2014a) ``Relative profit maximization and equivalence of Cournot and Bertrand equilibria in asymmetric duopoly,'' \textit{Economics Bulletin}, \textbf{34}, pp. 819-827.

\bibitem[Satoh and Tanaka (2014b)]{st} Satoh, A. and Y. Tanaka (2014b), ``Relative profit maximization in asymmetric oligopoly,'' \textit{Economics Bulletin}, \textbf{34}, pp. 1653-1664.

\bibitem[Satoh and Tanaka (2017)]{st17} Satoh, A. and Y. Tanaka (2017), ``Two person zero-sum game with two sets of strategic variables,'' MPRA Paper 73272, University Library of Munich, Germany. 

\bibitem[Sion (1958)]{sion} Sion, M. (1958), ``On general minimax theorems,'' \textit{Pacific Journal of Mathematics}, \textbf{8}, pp. 171-176.

\bibitem[Tanaka (2013a)]{eb1} Tanaka, Y. (2013a) ``Equivalence of Cournot and Bertrand equilibria in differentiated duopoly under relative profit maximization with linear demand,'' \textit{Economics Bulletin}, \textbf{33}, pp. 1479-1486. 

\bibitem[Tanaka (2013b)]{ebl1} Tanaka, Y. (2013b) ``Irrelevance of the choice of strategic variables in duopoly under relative profit maximization,'' \textit{Economics and Business Letters}, \textbf{2}, pp. 75-83.

\bibitem[Vega-Redondo(1997)]{redondo} Vega-Redondo, F. (1997) ``The evolution of Walrasian behavior,'', \textit{Econometrica}, \textbf{65}, pp. 375-384.
\end{thebibliography}
\end{document}